\documentclass[11pt]{article}
\usepackage[margin=1in]{geometry} 
\usepackage{amssymb,amsfonts,amsmath,bbm,mathrsfs,stmaryrd,mathtools,mathabx}
\usepackage{xcolor}
\usepackage{url}

\usepackage[shortlabels]{enumitem}
\usepackage{tensor}
\usepackage{braket}

\usepackage[T1]{fontenc}

\usepackage[colorlinks,
linkcolor=black!75!red,
citecolor=blue,
pdftitle={},
pdfproducer={pdfLaTeX},
pdfpagemode=None,
bookmarksopen=true,
bookmarksnumbered=true,
backref=page]{hyperref}

\usepackage{tikz}
\usetikzlibrary{arrows,calc,decorations.pathreplacing,decorations.markings,intersections,shapes.geometric,through,fit,shapes.symbols,positioning,decorations.pathmorphing}

\usepackage{braket}

\usepackage[amsmath,thmmarks,hyperref]{ntheorem}
\usepackage{cleveref}

\creflabelformat{enumi}{#2#1#3}

\crefname{section}{Section}{Sections}
\crefformat{section}{#2Section~#1#3} 
\Crefformat{section}{#2Section~#1#3} 

\crefname{subsection}{\S}{\S\S}
\AtBeginDocument{%
  \crefformat{subsection}{#2\S#1#3}%
  \Crefformat{subsection}{#2\S#1#3}% 
}

\crefname{subsubsection}{\S}{\S\S}
\AtBeginDocument{%
  \crefformat{subsubsection}{#2\S#1#3}%
  \Crefformat{subsubsection}{#2\S#1#3}% 
}

\theoremstyle{plain}

\newtheorem{lemma}{Lemma}[section]

\newtheorem{theorem}[lemma]{Theorem}

\theoremstyle{plain}
\theoremnumbering{Alph}
\newtheorem{theoremN}{Theorem}

\theoremstyle{plain}
\theorembodyfont{\upshape}
\theoremsymbol{\ensuremath{\blacklozenge}}

\newtheorem{definition}[lemma]{Definition}
\newtheorem{example}[lemma]{Example}

\newtheorem{remark}[lemma]{Remark}
\newtheorem{remarks}[lemma]{Remarks}

\newtheorem{notation}[lemma]{Notation}

\crefname{definition}{definition}{definitions}
\crefformat{definition}{#2definition~#1#3} 
\Crefformat{definition}{#2Definition~#1#3} 

\crefname{ex}{example}{examples}
\crefformat{example}{#2example~#1#3} 
\Crefformat{example}{#2Example~#1#3} 

\crefname{exs}{example}{examples}
\crefformat{examples}{#2example~#1#3} 
\Crefformat{examples}{#2Example~#1#3} 

\crefname{remark}{remark}{remarks}
\crefformat{remark}{#2remark~#1#3} 
\Crefformat{remark}{#2Remark~#1#3} 

\crefname{remarks}{remark}{remarks}
\crefformat{remarks}{#2remark~#1#3} 
\Crefformat{remarks}{#2Remark~#1#3} 

\crefname{convention}{convention}{conventions}
\crefformat{convention}{#2convention~#1#3} 
\Crefformat{convention}{#2Convention~#1#3} 

\crefname{notation}{notation}{notations}
\crefformat{notation}{#2notation~#1#3} 
\Crefformat{notation}{#2Notation~#1#3} 

\crefname{table}{table}{tables}
\crefformat{table}{#2table~#1#3} 
\Crefformat{table}{#2Table~#1#3}

\crefname{lemma}{lemma}{lemmas}
\crefformat{lemma}{#2lemma~#1#3} 
\Crefformat{lemma}{#2Lemma~#1#3} 

\crefname{proposition}{proposition}{propositions}
\crefformat{proposition}{#2proposition~#1#3} 
\Crefformat{proposition}{#2Proposition~#1#3} 

\crefname{corollary}{corollary}{corollaries}
\crefformat{corollary}{#2corollary~#1#3} 
\Crefformat{corollary}{#2Corollary~#1#3} 

\crefname{theorem}{theorem}{theorems}
\crefformat{theorem}{#2theorem~#1#3} 
\Crefformat{theorem}{#2Theorem~#1#3} 

\crefname{enumi}{}{}
\crefformat{enumi}{#2#1#3}
\Crefformat{enumi}{#2#1#3}

\crefname{assumption}{assumption}{Assumptions}
\crefformat{assumption}{#2assumption~#1#3} 
\Crefformat{assumption}{#2Assumption~#1#3} 

\crefname{construction}{construction}{Constructions}
\crefformat{construction}{#2construction~#1#3} 
\Crefformat{construction}{#2Construction~#1#3} 

\crefname{equation}{}{}
\crefformat{equation}{(#2#1#3)} 
\Crefformat{equation}{(#2#1#3)}

\numberwithin{equation}{section}

\theoremstyle{nonumberplain}
\theoremsymbol{\ensuremath{\blacksquare}}

\newtheorem{proof}{Proof}

\newcommand\altpf[2]{\newtheorem{#1}{#2}}

\newcommand\bC{{\mathbb C}}

\newcommand\bG{{\mathbb G}}
\newcommand\bH{{\mathbb H}}

\newcommand\bR{{\mathbb R}}

\newcommand\bZ{{\mathbb Z}}

\newcommand\cC{{\mathcal C}}
\newcommand\cD{{\mathcal D}}

\newcommand\cL{{\mathcal L}}

\DeclareMathOperator{\id}{id}
\DeclareMathOperator{\im}{\mathrm{im}}

\DeclareMathOperator{\Rep}{\mathrm{Rep}}

\DeclareMathOperator{\spn}{\mathrm{span}}

\newcommand{\cat}[1]{\textsc{#1}}

\newcommand{\qedhere}{\mbox{}\hfill\ensuremath{\blacksquare}}

\newcommand{\xrightarrowdbl}[2][]{%
  \xrightarrow[#1]{#2}\mathrel{\mkern-14mu}\rightarrow
}

\title{Semisimplicity manifesting as categorical smallness}
\author{Alexandru Chirvasitu}

\begin{document}

\date{}

\newcommand{\Addresses}{{% additional braces for segregating \footnotesize
  \bigskip
  \footnotesize

  \textsc{Department of Mathematics, University at Buffalo}
  \par\nopagebreak
  \textsc{Buffalo, NY 14260-2900, USA}  
  \par\nopagebreak
  \textit{E-mail address}: \texttt{achirvas@buffalo.edu}

  % % \medskip
  % % 
  % % \textsc{Department of Mathematics, INSTITUTION}
  % % \par\nopagebreak
  % % \textsc{ADDRESS}
  % % \par\nopagebreak
  % % \textit{E-mail address}: \texttt{??}
  % % 

}}

\maketitle

\begin{abstract}
  For a compact group $\mathbb{G}$, the functor from unital Banach algebras with contractive morphisms to metric spaces with 1-Lipschitz maps sending a Banach algebra $A$ to the space of $\mathbb{G}$-representations in $A$ preserves filtered colimits. Along with this, we prove a number of analogues: one can substitute unitary representations in $C^*$-algebras, as well as semisimple finite-dimensional Banach algebras (or finite-dimensional $C^*$-algebras) for $\mathbb{G}$.

  These all mimic results on the metric-enriched finite generation/presentability of finite-dimensional Banach spaces due to Ad{\'a}mek and Rosick{\'y}. We also give an alternative proof of that finite presentability result, along with parallel results on functors represented by compact metric, metric convex, or metric absolutely convex spaces. 
\end{abstract}


\vspace{.5cm}


%\tableofcontents

%%%%%%%%%%%%%%%%%%%%%%%%%%%%%%%%%%%%%%%%%%%%%%%%%%%%%%%%%%%%%%%%%%%%%%%%%%%%%
%%%%%%%%%%%%%%%%%%%%%%%%%%%%%%%%%%%%%%%%%%%%%%%%%%%%%%%%%%%%%%%%%%%%%%%%%%%%%
\section*{Introduction}

The smallness in the title is that of objects $c\in \cC$ in a category, by now a mainstay of the category-theory literature. There are several variations on the theme, but the common essence is that functors of the form $\cC(c,-)$ (the symbol means morphisms in $\cC$ with domain $c$) are required to preserve ``sufficiently filtered'' colimits (\Cref{se:prel} contains a brief recollection of filtration in the relevant sense). 
\begin{itemize}
\item Frequently (e.g. \cite[\S I, p.84]{zbMATH01799497}, \cite[Paragraph following Theorem 12.2.2]{riehl_ht}, \cite[Deﬁnition 5.9]{zbMATH07461229}) to $\cC(c,-)$ preserving certain transfinite chains of morphisms in a given class (one often also speaks of {\it compact} objects in that context). 
\item In additive categories the term sometimes \cite[\S 3.5, Proposition 5.1 and sentence following it]{zbMATH03425823} means that $\cC(c,-)$ preserves arbitrary coproducts (so filtered colimits of split embeddings). 
\item Then there are the related notions of {\it finite generation} (\cite[Definition 1.67]{ar}: $\cC(c,-)$ preserves filtered colimits of {\it monomorphisms} \cite[Definition 7.32]{ahs}) or {\it finite presentability} (\cite[Definition 1.1]{ar}: $\cC(c,-)$ preserves {\it all} filtered colimits), and so forth.
\end{itemize}

We will here be concerned with {\it enriched} categories \cite[\S 3.3]{riehl_ht}, i.e. those where hom spaces have some additional structure (so that smallness or finite presentability or generation also extend appropriately \cite[\S 4]{zbMATH07504760}). Specifically, the enrichment is mostly over the category $\cat{CMet}$ of complete metric spaces with non-expansive maps as morphisms.

The goal here is to analyze a number of interrelated phenomena whereby objects possessing some manner of representation-theoretic rigidity also acquire, by virtue of it, the types of smallness properties discussed above. Concretely, those objects, in \Cref{se:sml}, are compact groups on the one hand and (finite-dimensional) semisimple Banach algebras on the other, with finite-dimensional $C^*$-algebras as a particular case thereof. 

The semisimplicity of the paper's title is the aforementioned representation-theoretic rigidity: representations, in either case and in any reasonable sense, decompose as sums of (automatically finite-dimensional) irreducible summands. Ultimately, this is a manifestation of a strong form of {\it amenability} \cite[\S 5]{john-coh} or, in a reformulation, of one's ability to {\it average} \cite{cg-average}:
\begin{itemize}
\item against the {\it Haar (probability) measure} \cite[\S 1, pp.7-10]{rob} for compact groups;
\item and against the {\it diagonal} (\cite[Definition 1.1.]{zbMATH03389730}, \cite[Definition 1.9.19]{dales}, \cite[\S 7.2]{dael}) or {\it separability idempotent} \cite[\S II.1, Definition of Separability following Proposition 1.1]{zbMATH03342104} $e\in A\otimes A$ for a semisimple Banach algebra $A$.
\end{itemize}

An aggregate of \Cref{th:cgpsmall,th:ssalgsmall} presents as follows. 

\begin{theoremN}\label{thn:gpalg}
  \begin{enumerate}[(1)]
  \item\label{item:thn:gpalg:gp} For a compact group $\bG$ the functor
    \begin{equation*}
      \left(\cat{BAlg}_{1,\le 1}:=\text{unital Banach algebras with contractions}\right)
      \xrightarrow{\quad \Rep(\bG,-) \quad}
      \cat{CMet}
    \end{equation*}
    preserves filtered colimits, as does the functor
    \begin{equation*}
      \left(\cC^*_1:=\text{unital $C^*$-algebras}\right)
      \xrightarrow{\quad \Rep^*(\bG,-) \quad}
      \cat{CMet},
    \end{equation*}
    with the asterisk denoting {\it unitary} representations.

  \item\label{item:thn:gpalg:alg} For a semisimple unital Banach algebra $B$ the functor
    \begin{equation*}
      \cat{BAlg}_{1,\le 1}
      \xrightarrow{\quad\cat{BAlg}_1(B,-) \quad}
      \cat{CMet}
    \end{equation*}
    preserves filtered colimits, as does the functor
    \begin{equation*}
      \cC^*_1
      \xrightarrow{\quad\cC^*_1(B,-) \quad}
      \cat{CMet}
    \end{equation*}
    if $B\in \cC^*_1$ is a finite-dimensional $C^*$-algebra.  \qedhere
  \end{enumerate}
\end{theoremN}

There are a few alternative arrangements of the complex of proofs in \Cref{se:sml}, emphasizing that either half of \Cref{thn:gpalg} is recoverable from the other. Either way, the substance of the argument fits into the following sketchy mold:
\begin{enumerate}[(a)]
\item A representation
  \begin{equation*}
    \bullet
    \xrightarrow{\quad\varphi\quad}
    A=\varinjlim_i A_i
  \end{equation*}
  into a filtered colimit (with $\bullet$ a group or an algebra, as appropriate) can be approximated in the appropriate sense by a {\it map} $\bullet\xrightarrow{\varphi_i}A_i$ for some $i$.
\item Nevertheless, $\varphi_i$ can be chosen  so as to {\it almost} satisfy the requisite multiplicativity.
\item Whence it is uniformly close to an {\it actual} representation $\bullet\xrightarrow{}A_i$, essentially by the aforementioned averaging techniques afforded by the Haar measure and/or separability idempotents: \cite[Proposition 4]{zbMATH03577502} for the compact-group case and \cite[Theorem 3.1]{zbMATH04063807} for Banach or $C^*$-algebras. 
\end{enumerate}

This last portion of the argument, incidentally, is an instance of what a vast literature has come to refer to as {\it Hyers-Ulam} (or sometimes {\it Hyers-Ulam-Rassias}) {\it stability} (\cite[Chapter 1, discussion following Theorem 3 on p.4]{zbMATH06740635} and that book's extensive references): maps that are close to satisfying various constraints (multiplicativity, affineness, linearity, etc.) are close in various ways to maps that actually satisfy said constraints. The phrase is motivated by said authors' work \cite{zbMATH03101068,zbMATH03046328,zbMATH03074215,zbMATH03618859} laying the foundations of the subject, and \cite{zbMATH03577502,zbMATH04063807} are both mentioned in the survey paper \cite{zbMATH00149467} on the topic. 

The general shape of \Cref{thn:gpalg} places the discussion in the same circle of ideas as the $\cat{CMet}$-enriched finite generation \cite[Proposition 7.6]{zbMATH07461229} and in fact presentability \cite[Theorem 3.1]{zbMATH07760731} of finite-dimensional Banach spaces. Having come thus close to the matter, \Cref{th:banfinpres} gives an alternative proof of said \cite[Theorem 3.1]{zbMATH07760731}, deducing finite presentability from finite generation and proving the latter by an infinite-dimensional-analytic-manifold argument that might be of some independent interest. The subsequent discussion is devoted to some variations on the theme; a paraphrased \Cref{th:fdimim}:

\begin{theoremN}
  In each of the following cases, the $\cat{CMet}$-valued functor represented by $K$, defined on the category $\cat{Ban}_{\le 1}$ of Banach spaces with contractions and restricted to morphisms taking values in $(\le D)$-dimensional subspaces for some positive integer $D$, preserves filtered colimits:
  \begin{enumerate}[(1)]
  \item $K$ is a compact metric space and the morphisms are non-expansive maps.

  \item $K$ is a compact metric convex space and the morphisms are affine non-expansive maps (discussion preceding \Cref{re:distmonad}).

  \item $K$ is compact metric absolutely convex space and the morphisms are absolutely affine non-expansive (\Cref{def:acvx}).  \qedhere
  \end{enumerate}
\end{theoremN}

%%%%%%%%%%%%%%%%%%%%%%%%%%%%%%%%%%%%%%%%%%%%%%%%%%%%%%%%%%%%%%%%%%%%%%%%%%%%%
\subsection*{Acknowledgements}

Thanks are due to M. Brannan for numerous enlightening exchanges and M. Reyes for invaluable help in gaining access to some of the cited literature.

This work is partially supported by NSF grant DMS-2001128. 

%%%%%%%%%%%%%%%%%%%%%%%%%%%%%%%%%%%%%%%%%%%%%%%%%%%%%%%%%%%%%%%%%%%%%%%%%%%%%
%%%%%%%%%%%%%%%%%%%%%%%%%%%%%%%%%%%%%%%%%%%%%%%%%%%%%%%%%%%%%%%%%%%%%%%%%%%%%
\section{Preliminaries}\label{se:prel}

While $C^*$-algebras will of course be complex, we allow the not-uncommon ambiguity (\cite[p.953]{zbMATH03879553}, \cite[p.2]{zbMATH07461229}, etc.) of either real or complex scalars for Banach spaces or algebras. Little of the material depends substantively on the choice, and whatever does is easily adapted from one setup to the other.

Some common symbols:

\begin{notation}\label{not:common}
  \begin{enumerate}[(1), wide=0pt]
  \item\label{item:not:common:cmet} $\cat{CMet}$ is the category of complete metric spaces with {\it short} (or {\it non-expanding} \cite[Definition 1.4.1]{bbi} with {\it non-expansive} as a variant, or {\it 1-Lipschitz} \cite[Definition 1.1]{zbMATH05114904}) maps as morphisms, i.e. those which do not increase distances:
    \begin{equation*}
      d_Y(fx,fx')\le d_X(x,x')
      ,\quad\forall x,x'\in X\quad\text{for}\quad
      (X,d_X)\xrightarrow{\quad f\quad}(Y,d_Y). 
    \end{equation*}
    This seems to be the most popular notion of morphism when metric spaces are treated categorically: see \cite[\S 1]{zbMATH03243151}, \cite[\S 1]{zbMATH07461229}, \cite[\S 6]{zbMATH07504760}, \cite[\S 1]{zbMATH07469564} and so on.

    Following \cite[\S 1]{zbMATH06916415} or \cite[\S 1]{zbMATH07469564}, say, when metric spaces are allowed infinite distances we decorate the respective category symbols with an `$\infty$' subscript: $\cat{CMet}_{\infty}$, say. Infinite distances are occasionally taken for granted anyway, as in \cite[Definition 1.1.1]{bbi}. 
    
  \item $\cat{Ban}$ is the category of Banach spaces and bounded linear maps:
    \begin{equation*}
      \cat{Ban}(E,F):=\cL(E,F):=\left\{E\xrightarrow{f}F\ |\ f\text{ bounded and linear}\right\}. 
    \end{equation*}
    $\cat{Ban}_{\le 1}$, on the other hand, has the same objects and maps of norm $\le 1$ as morphisms. The latter is the better-behaved and in wide use in the vast literature: \cite{poth}, \cite[Example 1.48]{ar}, $\cat{Ban}_1$ in \cite[\S 3]{zbMATH05954612} and much other literature \cite{zbMATH03879553,zbMATH03955758,zbMATH04004954,zbMATH00097447,zbMATH01563366}, $\mathfrak{B}_1$ in \cite[paragraph preceding Lemma 1.1]{zbMATH06465637}, \cite[p.2]{zbMATH07461229}, their references, etc.
    
  \item $\cat{BAlg}$ and $\cat{BAlg}_1$ are the categories of Banach (unital) algebras respectively with bounded, linear, multiplicative (unital) maps. Further `$\le 1$' subscripts indicate that we only consider short maps (as in $\cat{BAlg}_{1,\le 1}$).

    The norm is always assumed \cite[Definition 1.1.2 and subsequent discussion]{dael} to satisfy
    \begin{equation*}
      \|ab\|\le \|a\|\cdot \|b\|\quad\text{and}\quad \|1\|=1\text{ in the unital case}. 
    \end{equation*}

  \item $\cC^*_1$ is the category of unital $C^*$-algebras.

  \item For a Banach algebra $A$ we write $A^{\times}$ for the multiplicative group of invertible elements therein. If $A$ is additionally $C^*$, $U(A)\subset A^{\times}$ is its unitary group. Both of these are (typically infinite-dimensional) {\it Banach Lie groups} in the sense of \cite[Definition IV.I]{neeb-inf} or \cite[\S VI.5]{lang-fund}. 
  \end{enumerate}  
\end{notation}

We will frequently refer to {\it ($\kappa$-)filtered colimits} in various categories $\cC$ for cardinals $\kappa$. These are colimits of functors $\cD\xrightarrow{}\cC$ whose domain category $\cD$ is itself ($\kappa$-)filtered. This, in turn, means \cite[Remark 1.21 ]{ar} that $\cD$ is non-empty, sets of $<\kappa$ objects map to some common object (depending on the set), and sets of $<\kappa$ parallel morphisms $c\to c'$ are coequalized by some arrow $c'\to c''$. Just plain {\it filtered} means $\aleph_0$-filtered, i.e. all of the above for {\it finite} sets of objects and morphisms (see also \cite[\S IX.1]{zbMATH01216133}).

Unless specified otherwise, we will always assume the domain category $\cD$ is a poset with one morphism $d\to d'$ precisely when $d\le d'$. In that context the term {\it directed} (in place of {\it filtered}) is also in wide use (\cite[Example 11.28(4)]{ahs}, for instance).  

%%%%%%%%%%%%%%%%%%%%%%%%%%%%%%%%%%%%%%%%%%%%%%%%%%%%%%%%%%%%%%%%%%%%%%%%%%%%%
%%%%%%%%%%%%%%%%%%%%%%%%%%%%%%%%%%%%%%%%%%%%%%%%%%%%%%%%%%%%%%%%%%%%%%%%%%%%%
\section{Smallness phenomena: compact groups and semisimple algebras}\label{se:sml}

The following simple remark will be of repeated use either directly or in avatar forms, so is worth singling out.

\begin{lemma}\label{le:cpct2metric}
  Let $K$ be a compact Hausdorff space and $(X,d)=\varinjlim_i (X,d_i)$ a filtered colimit in $\cat{CMet}_{\infty}$. The canonical morphism
  \begin{equation}\label{eq:colimcont2contcolim}
    \varinjlim_i \cat{Cont}(K,X_i)
    \xrightarrow{\quad}
    \cat{Cont}(K,\varinjlim_i X_i)
    =
    \cat{Cont}(K,X)\in \cat{CMet}_{\infty}. 
  \end{equation}
  is an isometric embedding. 
\end{lemma}
\begin{proof}
  The non-expansivity of \Cref{eq:colimcont2contcolim} is immediate, and makes no use of compactness (indeed, it holds for any topological space $K$ whatsoever). What we have to prove, then, is that for
  \begin{equation*}
    \text{continuous }K\xrightarrow{\quad f,g\quad} X_i
  \end{equation*}
  we have
  \begin{equation}\label{eq:supinfsup}
    \sup_{p\in K}d(\iota_i f p, \iota_i g p)
    =
    \inf_{j\ge i}\sup_{p\in K}d(\iota_{ji} f p, \iota_{ji} g p)
  \end{equation}
  for the colimit structure maps
  \begin{equation*}
    X_i\xrightarrow{\quad\iota_i\quad}X
    \quad\text{and}\quad
    X_i\xrightarrow{\quad\iota_{ji}\quad}X_{j},\quad i\le j
  \end{equation*}
  (cf. \cite[condition 2. in the proof of Theorem 3.1]{zbMATH07760731}). The non-expansivity of all of the maps in the factorization
  \begin{equation*}
    \begin{tikzpicture}[auto,baseline=(current  bounding  box.center)]
      \path[anchor=base] 
      (0,0) node (l) {$X_i$}
      +(2,.5) node (u) {$X_j$}      
      +(4,0) node (r) {$X$}
      ;
      \draw[->] (l) to[bend left=6] node[pos=.5,auto] {$\scriptstyle \iota_{ji}$} (u);
      \draw[->] (u) to[bend left=6] node[pos=.5,auto] {$\scriptstyle \iota_j$} (r);
      \draw[->] (l) to[bend right=6] node[pos=.5,auto,swap] {$\scriptstyle \iota_i$} (r);
    \end{tikzpicture}
  \end{equation*}
  makes plain the `$\le$' half of \Cref{eq:supinfsup}. As for the opposite inequality `$\ge$', a $j$-indexed {\it net} \cite[Definition 11.2]{wil-top}
  \begin{equation*}
    p_j\in X_i
    ,\quad
    \sup_{p\in K}d(\iota_i f p, \iota_i g p)
    \le
    d(\iota_{ji} f p_j, \iota_{ji} g p_j)-\varepsilon
    \text{ for some fixed }\varepsilon>0
  \end{equation*}
  will have \cite[Theorem 17.4]{wil-top} a subnet converging to some $p_0\in X_i$, whence the absurd conclusion that
  \begin{equation*}
    \sup_{p\in K}d(\iota_i f p, \iota_i g p)
    \le
    d(\iota_{i} f p_0, \iota_{i} g p_0)-\varepsilon.
  \end{equation*}
\end{proof}

\begin{remark}\label{re:dini}
  The aforementioned \cite[Theorem 3.1, proof of (2)]{zbMATH07760731} appeals to {\it Dini's theorem} \cite[Theorem 12.1]{zbMATH05012871}, to the effect that pointwise-convergent monotone sequences of continuous functions on compact (metrizable, in that reference) spaces are uniformly convergent. There are appropriate generalizations \cite[Corollary 6]{MR3462046} applicable to \Cref{le:cpct2metric}, but the direct argument did not appear substantively longer. 
\end{remark}

As a follow-up to \Cref{le:cpct2metric}, we record a smallness result of sorts for compact spaces with respect to maps into Banach spaces. Apart from whatever intrinsic interest the statement might possess, the proof will illustrate techniques useful later. 

\begin{theorem}\label{th:cpct2ban}
  For a compact Hausdorff space $K$ the functor
  \begin{equation*}
      \cat{Ban}_{\le 1}
      \xrightarrow{\quad \cat{Cont}(K,-) \quad}
      \cat{CMet}
    \end{equation*}
    preserves filtered colimits.
\end{theorem}
\begin{proof}
  Consider a filtered colimit $E=\varinjlim_i E_i$ in $\cat{Ban}_{\le 1}$ with structure maps $E_i\xrightarrow{\iota_i}E$.  \Cref{le:cpct2metric} settles all but the {\it surjectivity} of the map
  \begin{equation*}
      \varinjlim_i \cat{Cont}(K,E_i)
      \lhook\joinrel\xrightarrow{\quad}
      \cat{Cont}\left(K,\varinjlim_i E_i\right)
      =
      \cat{Cont}(K,E),
    \end{equation*}
    which will thus be the focus of the proof. The claim, then, is that a continuous map $K\xrightarrow{\varphi}E$ is arbitrarily uniformly approximable by continuous maps $K\xrightarrow{\varphi_i} E_i$. 

    Consider, for each $i$ and fixed $\varepsilon>0$, the set-valued map (or {\it correspondence} \cite[Definition 17.1]{zbMATH05265624})
    \begin{equation*}
      K\ni p
      \xmapsto{\quad\Phi_{\varepsilon,i}\quad}
      \{x\in E_i\ |\ \|\varphi(p)-\iota_i x\|\le\varepsilon\}
      \subset E_i.
    \end{equation*}
    Each $\Phi_{\varepsilon,i}$ takes closed convex values, all non-empty if $i$ is sufficiently large (no matter how small an $\varepsilon>0$ we choose and fix to begin with). The usual triangle-inequality argument also shows that $\Phi_{\varepsilon,i}$ are {\it lower hemicontinuous} \cite[Definition 17.2]{zbMATH05265624} (or sometimes \cite[p.712]{zbMATH06329568} {\it semi}continuous): for every open $U\subseteq E_i$,
    \begin{equation*}
      \left\{p\in K\ |\ \Phi_{\varepsilon,i}(p)\cap U\ne\emptyset\right\}\text{ is open}.
    \end{equation*}
    {\it Michael selection} (\cite[Theorem 17.66]{zbMATH05265624} or \cite[Theorem 1.1]{zbMATH06329568}) then shows that there is a continuous map
    \begin{equation*}
      K\ni p
      \xmapsto{\quad}
      \varphi_i(p)\in \Phi_{\varepsilon,i}\subset E_i
    \end{equation*}
    for sufficiently large $i$ (again, for arbitrarily small $\varepsilon>0$ chosen initially). By construction, the composition $K\xrightarrow{\iota_i\varphi_i}$ will be $\varepsilon$-close to the original $\varphi$. 
\end{proof}

Throughout the sequel, {\it representations} of a topological group $\bG$ into a topological (mostly Banach) algebra $A$ are understood to be continuous morphisms $\bG\to A^{\times}$ (always with the obvious topology on the latter, i.e. that induced by the norm in the Banach case). A variant is that of {\it unitary} representations, where one restricts the codomain to the unitary group $U(A)$ of a $C^*$-algebra.

\begin{theorem}\label{th:cgpsmall}
  Let $\bG$ be a compact group. 
  \begin{enumerate}[(1)]
  \item\label{item:th:cgpsmall:banach} The functor
    \begin{equation*}
      \cat{BAlg}_{1,\le 1}
      \xrightarrow{\quad \Rep(\bG,-) \quad}
      \cat{CMet}
    \end{equation*}
    preserves filtered colimits.

  \item\label{item:th:cgpsmall:cast} The same goes for the functor \begin{equation*}
      \cC^*_1
      \xrightarrow{\quad \Rep^*(\bG,-) \quad}
      \cat{CMet},
    \end{equation*}
    where the asterisk denotes {\it unitary} representations. 
  \end{enumerate}  
\end{theorem}
\begin{proof}
  \Cref{le:cpct2metric} again frees us to focus on the surjectivity of the map
  \begin{equation*}
      \varinjlim_i \Rep(\bG,A_i)
      \lhook\joinrel\xrightarrow{\quad}
      \Rep\left(\bG,\varinjlim_i A_i\right)
      =
      \Rep(\bG,A),
    \end{equation*}
    as we will throughout the proof, for a colimit
    \begin{equation*}
      A_i\xrightarrow{\quad\iota_i\quad}A=\varinjlim_i A_i\text{ in the appropriate category}.
    \end{equation*}
    
  \begin{enumerate}[(1), wide=0pt]
  \item\label{item:th:cgpsmall:banach-pf} Given a representation $\bG\xrightarrow{\varphi}A$, \Cref{th:cpct2ban} provides {\it maps} $\bG\xrightarrow{\varphi_i}A_i$ (not necessarily representations, a priori) with $\bG\xrightarrow{\iota_i\varphi_i}A$ close to $\varphi$.

    % % We are assuming that $A=\varinjlim_i A_i$ in $\cat{BAlg}_1$ and hence also in $\cat{CMet}$. A morphism $\bG\xrightarrow{\varphi} A^{\times}$ is then arbitrarily uniformly approximable by continuous {\it maps} (not necessarily morphisms, a priori) $\bG\xrightarrow{\varphi_i} A_i$. To see this, consider, for each $i$ and fixed $\varepsilon>0$, the set-valued map (or {\it correspondence} \cite[Definition 17.1]{zbMATH05265624})
    % % \begin{equation*}
    % %   \bG\ni s
    % %   \xmapsto{\quad\Phi_{\varepsilon,i}\quad}
    % %   \{a\in A_i\ |\ \|\varphi(s)-a\|\le\varepsilon\}
    % %   \subset A_i.
    % % \end{equation*}
    % % Each $\Phi_{\varepsilon,i}$ takes closed convex values, all non-empty if $i$ is sufficiently large (no matter how small an $\varepsilon>0$ we choose and fix to begin with). The usual triangle-inequality argument also shows that $\Phi_{\varepsilon,i}$ are {\it lower hemicontinuous} \cite[Definition 17.2]{zbMATH05265624}: for every open $U\subseteq A_i$,
    % % \begin{equation*}
    % %   \left\{s\in\bG\ |\ \Phi_{\varepsilon,i}(s)\cap U\ne\emptyset\right\}\text{ is open}.
    % % \end{equation*}
    % % {\it Michael selection} (\cite[Theorem 17.66]{zbMATH05265624} or \cite[Theorem 1.1]{zbMATH06329568}) then shows that there is a continuous map
    % % \begin{equation*}
    % %   \bG\ni s
    % %   \xmapsto{\quad}
    % %   \varphi_i(s)\in \Phi_{\varepsilon,i}\subset A_i
    % % \end{equation*}
    % % for sufficiently large $i$ (again, for arbitrarily small $\varepsilon>0$ chosen initially). 
    % % 
    
    I now claim further that for arbitrary $\delta>0$, we can make $\varphi_i$ (large $i$) $\delta$-close to being morphisms, in the sense \cite[D\'efinition 1]{zbMATH03577502} that
    \begin{equation}\label{eq:phiiapproxmult}
      \sup_{s,t\in \bG}\left\|\varphi_i(st)-\varphi_i(s)\varphi_i(t)\right\|<\delta. 
    \end{equation}
    Indeed, a sufficiently small $\varepsilon$ in $\|\iota_i\varphi_i-\varphi\|<\varepsilon$ (and the assumed multiplicativity of $\varphi$) will ensure that the binary function
    \begin{equation}\label{eq:binfuncveci}
      \bG^2\ni (s,t)
      \xmapsto{\quad}
      \iota_i\varphi_i(st)-\iota_i\varphi(s)\iota_i\varphi_i(t)
    \end{equation}
    takes its values in the $\frac{\delta}{2}$-ball of $A$, whence
    \begin{equation}\label{eq:binfuncnrji}
      \|\iota_{ji}\varphi_i(st)-\iota_{ji}\varphi(s)\iota_{ji}\varphi_{ji}(t)\|<\delta,\quad \forall s,t\in \bG^2
    \end{equation}
    for large $j\ge i$: the left-hand side of \Cref{eq:binfuncnrji} converges uniformly (on $\bG^2$) to the norm of \Cref{eq:binfuncveci}. Now simply substitute $\bG\xrightarrow{\iota_{ji}\varphi_i} A_j$  for $\varphi_i$. 
    
    \Cref{eq:phiiapproxmult} in hand, \cite[Proposition 4]{zbMATH03577502}, whose metric estimates apply uniformly in $i$ to the $A_i$ embedded isometrically in the ambient $A$, ensures the existence of a {\it representation} $\bG\to A_i$ $\varepsilon$-close to $\varphi_i$ for arbitrarily small $\varepsilon$ provided $\delta>0$ was chosen sufficiently small to begin with. 
    
  \item We now have $*$-structures both on the algebras $A$ and $A_i$ and the spaces $\cC(\bG,-)$ of continuous maps $\bG\xrightarrow{}A$:
    \begin{equation*}
      \cC(\bG,A)\ni \varphi
      \xmapsto{\quad}
      \varphi^*
      :=
      \left(\bG\ni s\xmapsto{\quad} \varphi(s^{-1})^*\in A\right),
    \end{equation*}
    as familiar in the context of $*$-algebras attached to (locally compact) groups (\cite[\S 1.4]{mart-proj}, \cite[Example VI.1.2]{tak2}, \cite[\S II.10.1.3]{blk}, \cite[\S 7.1.2]{ped-aut}, etc.).

    The unitary representations of $\bG$ into a $C^*$-algebra are precisely the self-adjoint ones for said $*$-structure on $\Rep(\bG,A)$ (indeed, this is what justifies the notation $\Rep^*(\bG,A)$ for the space of unitary representations). The argument in part \Cref{item:th:cgpsmall:banach-pf} simply goes through, taking care to work with self-adjoint maps throughout. When first approximating $\varphi\in \Rep^*(\bG,A)$ with $\varphi_i$, for instance, one can always replace an arbitrary $\varphi_i$ with its ``real part'' $\frac{\varphi_i+\varphi_i^*}2$ to ensure self-adjointness and then run through (the proof of) \cite[Proposition 4]{zbMATH03577502} similarly enforcing self-adjointness.
  \end{enumerate}
  This completes the proof as a whole. 
\end{proof}

There is an algebra-representation version of \Cref{th:cgpsmall}, the second part of which generalizes \cite[Proposition 3.5]{zbMATH07595194} (from filtered colimits of $C^*$ {\it embeddings} to filtered colimits period). Unbeknownst to us at the time, \cite[Proposition 3.5]{zbMATH07595194} is also (essentially) \cite[Lemma 2.3]{zbMATH03416517}: the latter is phrased in terms of colimits of embedding {\it sequences} $A_1\lhook\joinrel\xrightarrow{}A_2\lhook\joinrel\xrightarrow{}\cdots$, but the proof goes through virtually unchanged for arbitrary filtered colimits.

\begin{theorem}\label{th:ssalgsmall}
  \begin{enumerate}[(1)]
  \item\label{item:th:ssalgsmall:ban} For a semisimple unital Banach algebra $B\in\cat{BAlg}_1$ the functor 
    \begin{equation*}
      \cat{BAlg}_{1,\le 1}
      \xrightarrow{\quad\cat{BAlg}_1(B,-) \quad}
      \cat{CMet}
    \end{equation*}
    preserves filtered colimits.
    
  \item\label{item:th:ssalgsmall:cast} The same goes for the functor
    \begin{equation*}
      \cC^*_1
      \xrightarrow{\quad\cC^*_1(B,-) \quad}
      \cat{CMet}
    \end{equation*}
    if $B\in \cC^*_1$ is a finite-dimensional $C^*$-algebra.
  \end{enumerate}
\end{theorem}
\begin{proof}
  We deduce the two statements from their respective compact-group counterparts in \Cref{th:cgpsmall}. And as in that earlier result, it is enough to argue that a morphism
  \begin{equation}\label{eq:btolimai}
    B\xrightarrow{\quad\varphi\quad}A=\varinjlim_i A_i
  \end{equation}
  into a filtered colimit of isometric embeddings is approximable by morphisms into the individual $A_i$. 
  
  Because $B$ is a finite-dimensional semisimple real algebra, it is a finite product of matrix algebras over $\bR$, the quaternion division ring $\bH$, or $\bC$ \cite[Theorems 3.5 and 13.12]{lam}. It follows from this that $B$ is spanned by a compact subgroup $\bG\le B^{\times}$ (indeed, this {\it characterizes} semisimple real algebras among finite-dimensional ones \cite[Lemma 3.8]{cp-us}). From \Cref{th:cgpsmall} we have approximability of
  \begin{equation}\label{eq:phiphii}
    \bG \xrightarrow{\quad\varphi|_{\bG}\quad} A
    \quad\text{by}\quad
    \bG \xrightarrow{\quad\varphi_i\quad} A_i,
  \end{equation}
  so it will be enough to argue that said $\varphi_i$ extend across to $B$ linearly. Since $\varphi|_{\bG}$ does so extend by assumption, the conclusion follows from the fact that being sufficiently close, the two maps \Cref{eq:phiphii} are mutual conjugates: by $A^{\times}$ \cite[Corollaire to Proposition 4]{zbMATH03577502} in \Cref{item:th:ssalgsmall:ban} and by the unitary group $U(A)$ \cite[Proposition 1.7]{mart-proj} in \Cref{item:th:ssalgsmall:cast}.
\end{proof}

Deducing \Cref{th:ssalgsmall} from \Cref{th:cgpsmall} was a matter of choice rather than necessity. To emphasize the ties between the two branches of the discussion, it is perhaps worth noting that the roles can be reversed, proving the former {\it first} and {\it deducing} the compact-group version. 

\altpf{pf:th:ssalgsmall:alt}{Alternative proof of \Cref{th:ssalgsmall}}
\begin{pf:th:ssalgsmall:alt}
  We once more take it for granted that what needs proving is the arbitrary approximability (in the norm of $\cL(B,A)$) by morphisms $B\to A_i$. 
  
  The argument will parallel that in the proof of \Cref{th:cgpsmall}, applying the ``Newton-approximation'' method in \cite[proof of Theorem 3.1 and subsequent discussion]{zbMATH04063807} to first push $\varphi$ to a linear map $B\to A_i$ close to being a morphism, and then applying said approximation method to further deform into an algebra morphism.

  \begin{enumerate}[(I),wide=0pt]
  \item {\bf : Approximating algebra morphisms $B\to A$ by linear maps $B\to A_i$.} This is indeed possible to do arbitrarily well: for sufficiently large $i$, some linear $B\xrightarrow{\varphi_i}A_i$ will be $\delta'$-close to $B\xrightarrow{\varphi}A$ no matter how small $\delta'$ is.

    In metric-enriched-category language the claim is that of {\it approximate $\aleph_0$-generation} \cite[Deﬁnition 5.16]{zbMATH07461229} of finite-dimensional Banach spaces with respect to isometries, and it is proven in \cite[Proposition 7.6]{zbMATH07461229} (in its formally stronger form, in $\cat{Ban}_{\le 1}$ rather than $\cat{Ban}$).
    
    Furthermore, in case \Cref{item:th:ssalgsmall:cast} of the theorem we can assume the approximating $\varphi_i$ are self-adjoint (because $\varphi$ was) for the usual $*$-structure
    \begin{equation*}
      \psi^*(b):=\psi(b^*)^*\text{ for $b\in B$ and linear $B\xrightarrow{\psi}A$ or $A_i$}
    \end{equation*}
    by simply replacing $\varphi_i$ with their real parts $\frac{\varphi_i+\varphi_i^*}2$.
        
  \item {\bf : Further deforming into algebra morphisms.} Being $\delta'$-close to the original $\varphi$, a linear map $B\xrightarrow{\ \psi=\varphi_i\ }A_i$ fixed throughout the rest of the proof can be assumed $\delta$-close to being multiplicative as in \cite[\S 1]{zbMATH04063807} for arbitrarily small $\delta$: $\|\widecheck{T}\|<\delta$ in the usual sense of norming {\it multi}linear maps \cite[\S A.3, following Theorem A.3.35]{dales}, where
    \begin{equation*}
      \widecheck{T}(x,y):=T(xy)-T(x)t(y),\quad x,y\in B. 
    \end{equation*}
    Being (finite-dimensional) semisimple $B$ has (\cite[Theorem 1.9.21]{dales}, \cite[Theorem 7.2.1]{dael}) a {\it diagonal} \cite[\S 7.2]{dael} (which can be chosen self-adjoint in the $C^*$ case, by simply substituting $\frac{e+e^*}2$): an element
    \begin{equation*}
      e=\sum_\ell e'_\ell\otimes e''_\ell\in B\otimes B,\quad
      B\otimes B\ni e\xmapsto{\text{multiplication}}1\in B
      \quad\text{and}\quad
      be=eb,\ \forall b\in B.
    \end{equation*}
    The aforementioned approximation technique of \cite[Theorem 3.1]{zbMATH04063807} then proceeds by substituting for $\psi$ a series
    \begin{equation}\label{eq:psiser}
      (\psi_0:=\psi)+\psi_1+\psi_2+\cdots = \lim_n \left(\Psi_n:=\psi_0+\cdots+\psi_n\right),
    \end{equation}
    with each new term $\psi_{n+1}$ recoverable from the previous partial sum $\Psi_n$ by ``averaging'' against the integral $e$:
    \begin{equation*}
      \psi_{n+1}:= \sum_\ell \Psi_n(e'_{\ell})\widecheck{\Psi_n}(e''_{\ell},-),\quad\text{or}\quad      
      \begin{tikzpicture}[auto,baseline=(current  bounding  box.center)]
        \path[anchor=base] 
        (0,0) node (l) {$B$}
        +(2,.5) node (u) {$B^{\otimes 3}$}
        +(4,0) node (r) {$A_i$.}
        ;
        \draw[->] (l) to[bend left=6] node[pos=.5,auto] {$\scriptstyle e\otimes-$} (u);
        \draw[->] (u) to[bend left=6] node[pos=.5,auto] {$\scriptstyle \Psi_n\otimes \widecheck{\Psi_n}$} (r);
        \draw[->] (l) to[bend right=6] node[pos=.5,auto,swap] {$\scriptstyle \psi_{n+1}$} (r);
      \end{tikzpicture}
    \end{equation*}
    That argument functions to produce {\it (Banach or $C^*$-)algebra} morphisms \Cref{eq:psiser} $\varepsilon$-close to the original merely linear $\psi_0=\psi$, provided only the $\delta>0$ was sufficiently small to begin with; the estimates in question ($\varepsilon$, $\delta$, etc.) can be chosen uniformly no matter which $A_i$ we work in.
  \end{enumerate}
\end{pf:th:ssalgsmall:alt}

As for the announced implication reversal between the two theorems:

\altpf{pf:th:cgpsmall:alt}{Proof of \Cref{th:cgpsmall} assuming \Cref{th:ssalgsmall}}
\begin{pf:th:cgpsmall:alt}
  More precisely, parts \Cref{item:th:cgpsmall:banach} and \Cref{item:th:cgpsmall:cast} of \Cref{th:cgpsmall} reduce to their respective counterparts in \Cref{th:ssalgsmall} once we observe that the given representation $\bG\to A$ (unitary in \Cref{item:th:cgpsmall:cast}) factors through a finite-dimensional semisimple Banach (or $C^*$) quotient of the {\it convolution algebra} \cite[pp.53-55]{rob} $L^1(\bG)$. 

  This last claim, in turn, follows from the fact that $\bG\xrightarrow{\varphi}A^{\times}$ induces a {\it norm-}continuous representation of $\bG$ on
  \begin{itemize}
  \item the Banach space $A$ in \Cref{item:th:cgpsmall:banach}, with $A^{\times}$ acting thereon by left multiplication;
  \item or a Hilbert space $H$ for $A$ embedded as a closed $*$-subalgebra of $\cL(H)$ in \Cref{item:th:cgpsmall:cast}.
  \end{itemize}
  Said norm continuity implies (by \cite[Theorem 2.10]{2310.15139v1} in general, also \cite[p.257]{zbMATH03289866} for the unitary case) that those representations have finitely many isotypic components, hence the conclusion. 
\end{pf:th:cgpsmall:alt}

\begin{remarks}\label{res:semiproj}
  \begin{enumerate}[(1),wide=0pt]
  \item\label{item:res:semiproj:semiproj} \cite[Corollary 2.30]{zbMATH03996430} shows in particular that finite-dimensional $C^*$-algebras are {\it semiprojective} (\cite[Definition 2.10]{zbMATH03996430}, \cite[Definition II.8.3.7]{blk}). In the language of the present paper, this would be phrased as the requirement that
    \begin{equation*}
      \cC^*_1\xrightarrow{\quad \cC^*_1(B,-)\quad}\cat{Set}
    \end{equation*}
    preserve colimits of {\it surjection sequences}, with the entire discussion typically restricted to {\it separable} $C^*$-algebras \cite[\S II.8.3.9]{blk}.

    The arguments delivering \cite[Corollary 2.30]{zbMATH03996430} would presumably also go through for arbitrary filtered colimits; in any case, that result is a kind of opposite end of the spectrum to \cite[Proposition 3.5]{zbMATH07595194}: where the latter is concerned with embeddings, the former deals with surjections instead. \Cref{th:ssalgsmall}\Cref{item:th:ssalgsmall:cast} (or rather \Cref{th:liftepi}\Cref{item:th:liftepi:cast}) encompasses both (for {\it finite}-dimensional $B$; \cite[Corollary 2.30]{zbMATH03996430} proves semiprojectivity for wider classes of $C^*$-algebras).
    
  \item\label{item:res:semiproj:proj} There is also the stronger notion \cite[Definition 2.1]{zbMATH03996430} of {\it projectivity} for $B\in \cC^*_1$: the property that morphisms with domain $B$ lift along surjections.  

    As the name suggests, this is equivalent to category-theoretic projectivity \cite[\S 9.27 and Definition 9.1]{ahs}: the property that morphisms with the respective domain lift along {\it epimorphisms} \cite[Definition 7.39]{ahs}. Indeed, the epimorphisms in $\cC^*_1$ are precisely the surjections \cite[Proposition 2]{zbMATH03303492}.

    Projectivity truly is the stronger notion. $\cC^*_1$-morphisms out of $\bC^2$ simply pick out projections and, as noted (say) in \cite[paragraph preceding Proposition 2.5]{zbMATH02131689}, projections do not generally lift along surjections (consider e.g. the surjection dual to the inclusion of a disconnected compact Hausdorff space into a connected one). $\bC^2$ is thus not projective, but it is semiprojective by \cite[Proposition 2.18.]{zbMATH03996430}.
  \end{enumerate}
\end{remarks}

\Cref{res:semiproj}\Cref{item:res:semiproj:semiproj} suggests natural augmentations to \Cref{th:cgpsmall,th:ssalgsmall}: actual lifts (as opposed to just {\it approximate} lifts) when the connecting maps of the relevant colimits are surjections. We record the result here, with part \Cref{item:th:liftepi:cast} recovering (the finite-dimensional branch of) \cite[Corollary 2.30]{zbMATH03996430} by different means, and divorced from some of the unnecessary assumptions such as separability and the assumption that the colimits in question are along chained {\it sequences} of surjections. 

\begin{theorem}\label{th:liftepi}
  \begin{enumerate}[(1)]
  \item\label{item:th:liftepi:ban} The functors  
    \begin{equation*}
      \cat{BAlg}_{1,\le 1}
      \xrightarrow{\quad \Rep(\bG,-) \quad}
      \cat{Set}
      ,\quad
      \text{$\bG$ a compact group}
    \end{equation*}
    and
    \begin{equation*}
      \cat{BAlg}_{1,\le 1}
      \xrightarrow{\quad\cat{BAlg}_1(B,-) \quad}
      \cat{Set}
      ,\quad
      \text{$B\in\cat{BAlg}$ semisimple}.
    \end{equation*}
    preserve filtered colimits of surjective morphisms. 
  \item\label{item:th:liftepi:cast} The analogous statements hold in the $C^*$ setting, for
    \begin{equation*}
      \cC^*_1
      \xrightarrow{\quad \Rep^*(\bG,-) \quad}
      \cat{Set}
      ,\quad
      \text{$\bG$ a compact group}
    \end{equation*}
    and
    \begin{equation*}
      \cC^*_1
      \xrightarrow{\quad\cC^*_1(B,-) \quad}
      \cat{Set}
      ,\quad
      \text{$B\in\cC^*_1$ semisimple}.
    \end{equation*}
  \end{enumerate}
\end{theorem}
\begin{proof}
  The back-and-forth passage between the group and algebra pictures is easily effected, as in the alternative proofs of \Cref{th:cgpsmall,th:ssalgsmall}, each relying on the other. It will thus be enough (and convenient, for the purpose of fixing the notation) to focus on the group versions.

  Consider, then, a representation $\bG\xrightarrow{\varphi}A^{\times}$ for a filtered colimit
  \begin{equation*}
    A_i\xrightarrowdbl[\text{onto}]{\quad\iota_i\quad} A:=\varinjlim_i A_i.
  \end{equation*}
  \Cref{th:cgpsmall} provides representations $\bG\xrightarrow{\varphi_i}A$ with $\iota_i\varphi$ $\varepsilon$-close to $\varphi$ for arbitrarily small $\varepsilon>0$. Now, by \cite[Corollaire to Proposition 4]{zbMATH03577502} (or rather its proof via \cite[Lemme 7]{zbMATH03577502}), for arbitrarily small $\varepsilon'>0$, we can find $\varepsilon>0$ sufficiently small to ensure that
  \begin{equation}\label{eq:iphiconj}
    \iota_i\varphi_i = u\cdot\varphi\cdot u^{-1},\quad \text{some }u\in A^{\times}\text{ with }\|u-1\|<\varepsilon'.
  \end{equation}
  Surjectivity will allow lifting $u\in A^{\times}$ to some $u_i\in A_i$, and
  \begin{equation*}
    \|u_j-1\|<2\varepsilon'\text{ for large }j\ge i,\quad u_j:=\iota_{ji}(u_i)\in A_j.
  \end{equation*}
  $u_j$ will in particular be invertible if $\varepsilon'$ is small enough ($2\varepsilon'<1$ will do \cite[\S VII.2, Lemma 2.1]{con_course-2e}), and \Cref{eq:iphiconj} shows that
  \begin{equation*}
    \iota_j\circ\left(u_j^{-1}\cdot\varphi_j\cdot u_j\right) = \varphi
    \quad\text{for}\quad
    \begin{tikzpicture}[auto,baseline=(current  bounding  box.center)]
      \path[anchor=base] 
      (0,0) node (l) {$\bG$}
      +(2,.5) node (u) {$A_i$}
      +(4,0) node (r) {$A_j$}
      ;
      \draw[->] (l) to[bend left=6] node[pos=.5,auto] {$\scriptstyle \varphi_i$} (u);
      \draw[->>] (u) to[bend left=6] node[pos=.5,auto] {$\scriptstyle \iota_{ji}$} (r);
      \draw[->] (l) to[bend right=6] node[pos=.5,auto,swap] {$\scriptstyle \varphi_j$} (r);
    \end{tikzpicture}
  \end{equation*}
  This all admits the obvious modifications in the $C^*$ case: one conjugates by {\it unitaries} in place of arbitrary invertibles, using \cite[Proposition 1.7]{mart-proj} instead of \cite[Lemme 7]{zbMATH03577502}, and passing (when and wherever necessary) from an invertible $x$ (in $A^{\times}$ or $A_{i}^{\times}$) to the corresponding unitary resulting from the {\it polar decomposition} \cite[\S II.3.2.9]{blk} $x=u|x|$.
\end{proof}

\section{Banach-space finite presentability via almost affine maps}\label{se:banfinpres}

The brief detour taken in the course of the alternative proof of \Cref{th:ssalgsmall} into the $\aleph_0$-generation material of \cite[\S 7]{zbMATH07461229} (and by extension, the generalizations thereof in \cite[\S 3]{zbMATH07760731}) suggests a closer examination of the links between the present material and those sources.

One aspect of this will be to recover the {\it $\cat{CMet}$-enriched finite presentability} of \cite[Theorem 3.1]{zbMATH07760731} with a different proof, illustrating a technique that could have been applied in \Cref{se:sml} as well. Recall first the result, very much analogous to \Cref{th:cpct2ban,th:cgpsmall,th:ssalgsmall}:

\begin{theorem}\label{th:banfinpres}
  For a finite-dimensional Banach space $F$ the functor
  \begin{equation*}
    \cat{Ban}_{\le 1}
    \xrightarrow{\quad\cat{Ban}_{\le 1}(F,-) \quad}
    \cat{CMet}
  \end{equation*}
  preserves filtered colimits. 
\end{theorem}
\begin{proof}
  As usual, let 
  \begin{equation}\label{eq:bandiltcolim}
    E_i\xrightarrow{\quad\iota_i\quad}E=\varinjlim_i E_i
  \end{equation}
  be a filtered $\cat{Ban}_{\le 1}$-colimit.

  The isometric embedding of \Cref{le:cpct2metric} goes through for the closed unit ball $K=F_{\le 1}$ of $F$ when restricting to morphisms of {\it totally convex spaces} (\cite[\S 2]{zbMATH03879553}, \cite[Example 1.48]{ar}), i.e. those satisfying
  \begin{equation*}
    \varphi\left(\sum_{n\in \bZ_{\ge 0}}a_n x_n\right)
    =
    \sum_{n\in \bZ_{\ge 0}}a_n \varphi(x_n)
    ,\quad
    \forall x_n\in K
    ,\quad
    \forall \sum|a_n|\le 1. 
  \end{equation*}
  This again allows us to focus attention on surjectivity: an arbitrary linear non-expansive $F\xrightarrow{\varphi}E$ belongs to the image of the canonical map
  \begin{equation}\label{eq:feifei}
    \varinjlim_i \cat{Ban}_{\le 1}(F,E_i)
    \lhook\joinrel\xrightarrow{\quad}
    \cat{Ban}_{\le 1}\left(F,\varinjlim_i E_i\right)
    =
    \cat{Ban}_{\le 1}(F,E).
  \end{equation}
  There is furthermore no loss in assuming
  \begin{equation}\label{eq:f2eemb}
    F\lhook\joinrel\xrightarrow{\quad\varphi\quad}E
  \end{equation}

  embedded isometrically via $\varphi$, and we henceforth will: we can always substitute the inclusion of $F/\ker\varphi$ with its metric inherited from the resulting embedding into $E$ for the original data. 
  
  The argument now branches.
  \begin{enumerate}[(I),wide=0pt]
  \item\label{item:th:banfinpres:isometr} {\bf : Isometric connecting maps $E_i\lhook\joinrel\xrightarrow{\iota_{ji}}E_j$.} This is nothing but the ($\cat{CMet}$-enriched) finite {\it generation} of \cite[Proposition 7.6]{zbMATH07461229}, to which we can simply appeal. Alternatively, the following argument might be of some independent interest.

    The realization of $E$ as the closed union of the $E_i\le E$ implies in particular that the unit ball $F_{\le 1}$ will be contained in the $\varepsilon$-neighborhood of the unit ball $E_{i,\le 1}$ for some $i$ (for arbitrarily small $\varepsilon$). This means in particular that the {\it Hausdorff distance} \cite[Definition 7.3.1]{bbi}
    \begin{equation*}
      d_H(F_{i,\le 1},F_{\le 1})
      :=
      \inf\left\{r>0\ |\ \right\}
    \end{equation*}
    is small for some $(d:=\dim F)$-dimensional subspace $F_i\le E_i$. In particular, as explained in \cite[p.17]{zbMATH03236278}, $F$ is in the closure of the subset 
    \begin{equation*}
      \left\{F_i\le E_i\ |\ \dim F_i=d:=\dim F,\ i\in I\right\}
    \end{equation*}
    of the {\it Grassmannian (Banach) manifold} (\cite[\S 2.1]{zbMATH03236278}, \cite[Example 2.21]{zbMATH07027314})
    \begin{equation*}
      G(d,E):=\left\{\text{$d$-dimensional subspaces of }E\right\}.
    \end{equation*}
    Per the same sources \cite[\S 2.1]{zbMATH03236278} and/or \cite[Example 2.21]{zbMATH07027314}, having fixed a supplement $E'$ of $F\le E$ in the sense that $E=F\oplus E'$, there is an identification
    \begin{equation*}
      \cL(F,E')\ni L
      \xmapsto{\quad}
      \left(\Gamma_L:=\text{graph of $L$}\right)\in G(d,E).
    \end{equation*}
    In particular, for $F_i$ as above sufficiently close to $F$, we have $F_i=\Gamma_{L_i}$ for small-norm linear $F\xrightarrow{L_i}E'$. It follows that $F\xrightarrow[\ \cong\ ]{1+L_i}F_i$ is a linear isomorphism close to $1$, so a small rescaling will turn it into a contraction $F\cong F_i\le E_i$ that approximates the original inclusion $F\le E$. 
    
    % % In other words, we propose to first recover the ($\cat{CMet}$-enriched) finite {\it generation} of \cite[Proposition 7.6]{zbMATH07461229}, again with a different proof.
    % % 
    % % The claim is now that the embedding \Cref{eq:f2eemb} is arbitrarily approximable with embeddings $F\lhook\joinrel\xrightarrow{}E_i$. $F$ will be arbitrarily close to $(d=\dim F)$-dimensional subspaces of $E_i$ in the {\it Grassmannian} \cite[preceding Lemma 3.12]{upm-ban} $G(d,E)$ of $d$-dimensional subspaces of $E$. 
    % % 
    
  \item\label{item:th:banfinpres:red2surj} {\bf : Reduction to the case of surjective $E_i\xrightarrowdbl{\iota_{ji}}E_j$.} Apply step \Cref{item:th:banfinpres:isometr} to the filtered diagram consisting of the embeddings
    \begin{equation*}
      \overline{E}_i:=\iota_i(E_i)\lhook\joinrel\xrightarrow{\quad}E
    \end{equation*}
    to conclude that $\varphi$ is well approximable with an embedding into one such $\overline{E_i}$, and then replace the original diagram with that consisting of the images of $E_i\xrightarrow{\iota_{ji}}E_j$ for $j\ge i$; the connecting maps are by construction surjections. 
    
  \item {\bf : Conclusion.} Having reduced the problem to filtered diagrams of {\it sur}jections, we relegate that case to \Cref{le:bansurjconn}.
  \end{enumerate}
\end{proof}

\begin{lemma}\label{le:bansurjconn}
  For finite-dimensional $F\in\cat{Ban}$ the functor
  \begin{equation*}
    \cat{Ban}_{\le 1}
    \xrightarrow{\quad \cat{Ban}_{\le 1}(F,-) \quad}
    \cat{CMet}
  \end{equation*}
  preserves filtered colimits of surjective morphisms. 
\end{lemma}
\begin{proof}
  As always, the issue is the surjectivity of \Cref{eq:feifei}, which however in the present context (of surjective connecting maps $E_i\xrightarrowdbl{\iota_{ji}}E_j$) is almost immediate. Harmlessly assuming $F\lhook\joinrel\xrightarrow{\varphi}E$ to be an isometric embedding (as in the proof of \Cref{th:banfinpres}), lift $F$ to some $F_i\le E_i$ mapped onto it isomorphically by $E_i\xrightarrowdbl{\iota_i}E$ (possible, by surjectivity and the finite-dimensionality of $F$), and replace the original diagram
  \begin{equation*}
    \left(E_j\xrightarrow{\quad\iota_{j'j}}E_{j'}\right)_{j\le j'}
    \quad\text{with}\quad
    \left(F_j\xrightarrow{\quad\iota_{j'j}}F_{j'}\right)_{i\le j\le j'},\quad F_j:=\iota_{ji}(F_i). 
  \end{equation*}
  The colimit
  \begin{equation*}
    F_j\xrightarrow{\quad\iota_j\quad}F\cong \varinjlim_j F_j
  \end{equation*}
  now consists of linear isomorphisms whose inverses' norms approach $1$. Now take the desired approximate lift to be the inverse of some $\iota_j$ for $j$ large, slightly rescaled so as to ensure non-expansivity. 
\end{proof}

Note, incidentally, that the obvious parallel to \Cref{th:liftepi} does {\it not} hold in the context of \Cref{le:bansurjconn}: 
\begin{equation*}
  \cat{Ban}_{\le 1}
  \xrightarrow{\quad \cat{Ban}_{\le 1}(F,-) \quad}
  \cat{Set}
\end{equation*}
does not, in general, preserve filtered colimits of surjective morphisms, for reasons evident in the very last step of the proof of \Cref{le:bansurjconn}: there is a scaling required to correct for possible expansivity. This is the case even if one restricts attention to sequential chains of linear isomorphisms:

\begin{example}\label{ex:seqisos}
  By the usual (e.g. \cite[Propositions 1.1.6 and 1.1.8]{zbMATH00949298}) correspondence between norms and their unit balls, a Banach space structure on $\bR^2$ is determined by the convex, origin-symmetric convex body that is to be that structure's unit ball, and any such will do. There is thus a colimit
  \begin{equation*}
    \bR^2=: F_i\xrightarrow{\quad \iota_i:=\id\quad} F=\varinjlim_i F_i\cong \bR^2
  \end{equation*}
  with the unit balls of $F_i$ being, respectively, the convex hulls of the $(2i)^{th}$ roots of unity and that of $F$ being the usual unit disk. Naturally, $F\xrightarrow{\id}F$ does not split through any {\it non-expansive} linear $F\to F_i$. 
\end{example}

\begin{remark}\label{re:epimonosplitalgs}
  As alluded to above, the same strategy (as in \Cref{th:banfinpres}) could have been employed in \Cref{th:cgpsmall,th:ssalgsmall}, of settling the isometric- and surjective-connecting-map cases separately and conjoining them. The $C^*$ case, for instance, could have been pieced together from
  \begin{itemize}[wide]
  \item the isometric-embedding version of \Cref{th:ssalgsmall}\Cref{item:th:ssalgsmall:cast}, i.e. \cite[Proposition 3.5]{zbMATH07595194} or \cite[Lemma 2.3]{zbMATH03416517} (taking it for granted, in the latter, that the sequence statement carries over to arbitrary filtered colimits);
  \item and the surjective version, say \cite[Corollary 2.30]{zbMATH03996430}, again taking care to transport that discussion from sequences of chained $C^*$ surjections to arbitrary filtered colimits of such. 
  \end{itemize}
\end{remark}

By way of recycling the Grassmannian argument employed in the proof of \Cref{th:banfinpres} (part \Cref{item:th:banfinpres:isometr} of that proof) we list a few variants thereof after introducing the requisite notation.

First, for a positive integer $D$, a left-hand `$\le D$' subscript on a space of maps into a vector space indicates that each of the said maps takes values in a subspace of dimension $\le D$. Example, for a topological space $K$ and a topological vector space $E$:
\begin{equation*}
  \tensor[_{\le D}]{\cat{Cont}(K,E)}{}
  :=
  \left\{\text{continuous }K\xrightarrow{f}E\ |\ \dim\spn\im f\le D\right\}.
\end{equation*}

We will also consider (typically compact) metric spaces equipped with an abstract {\it convex structure}, i.e. what \cite[Definition 1]{zbMATH06183792} (following \cite[\S 2]{MR2799809}) refers to as a {\it convex-like structure}. The definition is guessable enough: the metric space $(X,d)$ also admits finite convex combinations, i.e. a ternary operation
\begin{equation*}
  [0,1]\times X\times X
  \ni
  (\lambda,x_0,x_1)
  \xmapsto{\quad}
  \lambda x_0+(1-\lambda)x_1
  \in
  X
\end{equation*}
with the expected arithmetic properties, plus metric compatibility in the sense that
\begin{equation*}
  d\left(\sum \lambda_i x_i,\ \sum\lambda_i x_i'\right)
  \le \sum\lambda_i d\left(x_i,x_i'\right)
  \quad\text{for}\quad
  x_i,x_i'\in X
  ,\quad
  \lambda_i\ge 0
  ,\quad
  \sum\lambda_i=1.
\end{equation*}
Recall \cite[Theorem 9]{zbMATH06183792} also that when {\it complete} such structures are precisely the closed convex subsets of real Banach spaces. We write $\cat{CMet}^{\cat{cvx}}$ for the resulting category of {\it complete metric convex spaces}, with (as expected) non-expansive {\it affine} maps $X\xrightarrow{\varphi}Y$ as morphisms:
\begin{equation*}
  \varphi(\lambda x_0+(1-\lambda)x_1) = \lambda\varphi(x_0)+(1-\lambda)\varphi(x_1)
  ,\quad
  \forall \lambda\in[0,1]
  ,\quad
  \forall x_i\in X. 
\end{equation*}

\begin{remark}\label{re:distmonad}
  The symbol $\cat{CMet}^{\cat{cvx}}$ is meant as reminiscent of the notation $\cC^T$ for the ({\it Eilenberg-Moore}) category of {\it $T$-algebras} \cite[\S VI.2]{zbMATH01216133} for a {\it monad} \cite[VI.1, Definition]{zbMATH01216133} on a category $\cC$.

  Metrics aside, abstract {\it convex set} \cite[Definition 3]{zbMATH05806299} are precisely \cite[Theorem 4]{zbMATH05806299} the algebras over the {\it distribution monad} (\cite[equation (4)]{zbMATH05806299}, \cite[\S 3.1]{zbMATH06814745}) $\cD$ on $\cat{Set}$:
  \begin{equation*}
    \begin{aligned}
      \cD X&:=\left\{\text{formal convex combinations of finite tuples in }X\right\}\\
           &\cong \left\{\text{finitely-supported probability measures on }X\right\}.
    \end{aligned}    
  \end{equation*}
  Said monad lifts straightforwardly to $\cat{Met}$ (plain metric spaces) and thence restricts to $\cat{CMet}$, by equipping $\cD X$ with the {\it Kantorovich-Rubinstein distance} \cite[Particular Case 6.2 of Definition 6.1 and Remark 6.5]{zbMATH05306371}:
  \begin{equation}\label{eq:krdist}
    d(\mu,\mu'):=\sup_{\mathrm{Lip}(f)\le 1}
    \left|\int_X f\ \mathrm{d}\mu - \int_X f\ \mathrm{d}\mu'\right|
    ,\quad
    \mathrm{Lip}(f):=\inf\left\{C>0\ |\ f\text{ is $C$-Lipschitz}\right\}.
  \end{equation}
  The aforementioned lifting means that we have diagrams
  \begin{equation*}
    \begin{tikzpicture}[auto,baseline=(current  bounding  box.center)]
      \path[anchor=base] 
      (0,0) node (l) {$\cat{CMet}$}
      +(2,.5) node (u) {$\cat{CMet}$}
      +(2,-.5) node (d) {$\cat{Set}$}
      +(4,0) node (r) {$\cat{Set}$}
      ;
      \draw[->] (l) to[bend left=6] node[pos=.5,auto] {$\scriptstyle \cD$} (u);
      \draw[->] (u) to[bend left=6] node[pos=.5,auto] {$\scriptstyle \cat{forget}$} (r);
      \draw[->] (l) to[bend right=6] node[pos=.5,auto,swap] {$\scriptstyle \cat{forget}$} (d);
      \draw[->] (d) to[bend right=6] node[pos=.5,auto,swap] {$\scriptstyle \cD$} (r);
    \end{tikzpicture}
    \quad\text{and hence}\quad
    \begin{tikzpicture}[auto,baseline=(current  bounding  box.center)]
      \path[anchor=base] 
      (0,0) node (l) {$\cat{CMet}^{\cD}$}
      +(2,.5) node (u) {$\cat{CMet}$}
      +(2,-.5) node (d) {$\cat{Set}^{\cD}$}
      +(4,0) node (r) {$\cat{Set}$}
      ;
      \draw[->] (l) to[bend left=6] node[pos=.5,auto] {$\scriptstyle \cat{forget}$} (u);
      \draw[->] (u) to[bend left=6] node[pos=.5,auto] {$\scriptstyle \cat{forget}$} (r);
      \draw[->] (l) to[bend right=6] node[pos=.5,auto,swap] {$\scriptstyle \cat{forget}$} (d);
      \draw[->] (d) to[bend right=6] node[pos=.5,auto,swap] {$\scriptstyle \cat{forget}$} (r);
    \end{tikzpicture}
  \end{equation*}
  of categories and functors, commutative up to the obvious natural isomorphisms. To come back full circle, it will be a simple exercise to verify that the resulting Eilenberg-Moore category $\cat{CMet}^{\cD}$ is nothing but $\cat{CMet}^{\cat{cvx}}$; the preference for `$\cat{cvx}$' over `$\cD$' stems from the former's being somewhat more descriptive. 
\end{remark}

{\it Mutatis mutandis}:

\begin{definition}\label{def:acvx}
  The monad $\cat{Set}\xrightarrow{\cat{acvx}}\cat{Set}$ associates to a set its collection of finitely-supported {\it complex}-valued measures $\mu$ of {\it total variation} \cite[\S 6.1]{zbMATH01022658} $\le 1$:
  \begin{equation*}
    \sharp\{x\in X\ |\ \mu(x)\ne 0\}<\infty
    \quad\text{and}\quad
    \sum_X|\mu(x)|\le 1.
  \end{equation*}
  The monad lifts to (complete) metric spaces much as the distribution monad did, with \Cref{eq:krdist} as the distance between measures.

  The categories of (complete, metric) {\it absolutely convex spaces} are the resulting Eilenberg-Moore categories $\cat{Set}^{\cat{acvx}}$, $\cat{Cmet}^{\cat{acvx}}$, etc.

  We will occasionally also refer to morphisms of absolutely convex spaces as {\it absolutely affine}, by analogy to plain convex structures. 
\end{definition}

\begin{remarks}\label{res:acvx}
  \begin{enumerate}[(1), wide=0pt]
  \item The (plain, non-metric) absolutely convex spaces are the {\it absolutely convex modules} of \cite[\S 2]{zbMATH01563366} and {\it finitely totally convex spaces} of \cite[Definition 2.9]{zbMATH03879553}. They are equipped with binary operations
    \begin{equation*}
      X\times X
      \ni (x_0,x_1)
      \xmapsto{\quad}
      a_0 x_0+a_1x_1
      \in X
      \quad\text{for}\quad |a_0|+|a_1|\le 1
    \end{equation*}
    satisfying the expected arithmetic constraints. 

  \item The terminology matches that in common use in the literature on topological vector spaces: \cite[\S 15.10]{koeth_tvs-1} or \cite[\S I.1, discussion preceding Lemma 1]{rr_tvs}, say. Per the latter source, the absolutely convex structure can be recovered from plain convexity together with the scaling operations $x\mapsto{} ax$, $|a|\le 1$ (sets closed under said scaling are sometimes called {\it balanced}).
  \end{enumerate}
\end{remarks}

\begin{theorem}\label{th:fdimim}
  \begin{enumerate}[(1)]
  \item\label{item:th:fdimim:met} For a compact metric space $(K,d_K)$ and a positive integer $D$ the functor
    \begin{equation*}
      \cat{Ban}_{\le 1}
      \xrightarrow{\quad\tensor[_{\le D}]{\cat{CMet}}{}(K,-) \quad}
      \cat{CMet}
    \end{equation*}
    preserves filtered colimits. 

  \item\label{item:th:fdimim:cvmet} For a compact metric convex space $(K,d_K)$ and a positive integer $D$ the functor
    \begin{equation*}
      \cat{Ban}_{\le 1}
      \xrightarrow{\quad\tensor[_{\le D}]{\cat{CMet}}{^{\cat{cvx}}}(K,-) \quad}
      \cat{CMet}
    \end{equation*}
    preserves filtered colimits.

  \item\label{item:th:fdimim:acvmet} For a compact metric absolutely convex space $(K,d_K)$ and a positive integer $D$ the functor
    \begin{equation*}
      \cat{Ban}_{\le 1}
      \xrightarrow{\quad\tensor[_{\le D}]{\cat{CMet}}{^{\cat{acvx}}}(K,-) \quad}
      \cat{CMet}
    \end{equation*}
    preserves filtered colimits. 
  \end{enumerate}  
\end{theorem}
\begin{proof}  
  As observed repeatedly, the fact that the canonical map from the colimit of hom spaces to the hom space into the colimit is an isometric embedding is immediate by restricting \Cref{le:cpct2metric} to the appropriate spaces of maps (non-expansive, affine, etc.). What is at issue, rather, is the surjectivity of
  \begin{equation*}
    \varinjlim_i \cC(K,E_i)
    \xrightarrow{\quad}
    \cC\left(K,\varinjlim_i E_i\right)
    =
    \cC(K,E)
  \end{equation*}
  for the various categories $\cC$ and a Banach-space filtered colimit \Cref{eq:bandiltcolim}.
  
  % % We now follow the same strategy as in \Cref{th:banfinpres}, assuming throughout that $K\lhook\joinrel\xrightarrow{\varphi}E$ is an isometric embedding (also affine or absolutely affine, as needed), factoring through a $D$-dimensional subspace of $E$. 
  % % 
  
  % % \begin{description}[wide=0pt]
  % % \item {\bf Reduction to surjective connecting maps.} This carries through precisely as in the proof of \Cref{th:banfinpres}.
  % %   
  % %   Because we are assuming that $\dim\spn K\le D$, the Grassmannian argument in step \Cref{item:th:banfinpres:isometr} shows that for large $i$, an invertible operator on $E$ arbitrarily close to $1$ will map $K$ into one of the subspaces $E_i\le E$ when the connecting maps are assumed isometric. Step \Cref{item:th:banfinpres:red2surj} of the same proof then goes through verbatim.
  % %   
  % %   We will henceforth assume that the fixed colimit \Cref{eq:bandiltcolim} is one of surjections ({\it filtered quotients} is a convenient short-hand phrase). 
  % %   
  % % \item {\bf Proof of \Cref{item:th:fdimim:cvmet}, assuming \Cref{item:th:fdimim:met}.}
  % %   
  % %   {\color{magenta} do}
  % %   
  % % \item {\bf Proof of \Cref{item:th:fdimim:acvmet}, assuming \Cref{item:th:fdimim:met}.}
  % %   
  % %   {\color{magenta} do}
  % %   
  % % \item {\bf Conclusion.} We have reduced the problem to proving \Cref{item:th:fdimim:met} for filtered quotients.
  % %   
  % %   {\color{magenta} do}    
  % % \end{description}
  % % This finishes the proof. 
  % % 
  
  All claims, at this point, follow from \Cref{th:banfinpres}: a morphism $K\xrightarrow{} E$ with the requisite affineness properties factors through the isometric embedding $\spn\varphi(K)\lhook\joinrel\xrightarrow{}E$, to which \Cref{th:banfinpres} applies.
\end{proof}

\begin{remark}\label{re:acmetricadj}
  \cite[\S 6]{zbMATH03879553} equips every absolutely convex space (i.e. object of $\cat{Set}^{\cat{acvx}}$, in the notation of \Cref{def:acvx}) with a {\it seminorm} $\|\cdot\|$: defined as one usually does \cite[\S 14.1]{koeth_tvs-1} on vector spaces, except that scaling by $\lambda$ is only allowed for $|\lambda|\l1 1$. This will make every such space ($X$, say) {\it pseudometric} \cite[Definition 2.1]{wil-top} via
  \begin{equation*}
    d(x,x') := \|x\| + \|x'\|,\ \forall x,x'\in X.
  \end{equation*}
  Upon identifying points not distinguished by that pseudometric, we obtain a convex {\it metric} space. All in all, we have an adjunction
  \begin{equation*}
    \begin{tikzpicture}[auto,baseline=(current  bounding  box.center)]
      \path[anchor=base] 
      (0,0) node (l) {$\cat{Set}^{\cat{acvx}}$}
      +(4,0) node (r) {$\cat{Met}^{\cat{acvx}}$}
      +(2,0) node (m) {$\bot$}
      ;
      \draw[->] (l) to[bend left=20] node[pos=.5,auto] {$\scriptstyle $} (r);
      \draw[->] (r) to[bend left=20] node[pos=.5,auto] {$\scriptstyle \cat{forget}$} (l);
    \end{tikzpicture}
  \end{equation*}
  (the tail of the $\top$ pointing towards the left adjoint, as customary \cite[Definition 19.3]{ahs}). Functoriality, for instance follows from the fact \cite[Lemma 6.3]{zbMATH03879553} that morphisms in $\cat{Set}^{\cat{acvx}}$ are automatically $\|\cdot\|$-contractive.

  One can then also complete distances to push the adjunction further into $\cat{CMet}^{\cat{acvx}}$. 
\end{remark}

%%%%%%%%%%%%%%%%%%%%%%%%%%%%%%%%%%%%%%%%%%%%%%%%%%%%%%%%%%%%%%%%%%%%%%%%%%%%%
%%%%%%%%%%%%%%%%%%%%%%%%%%%%%%%%%%%%%%%%%%%%%%%%%%%%%%%%%%%%%%%%%%%%%%%%%%%%%

\addcontentsline{toc}{section}{References}
%\bibliography{bib}{}

\begin{thebibliography}{10}

\bibitem{zbMATH01799497}
J.~Ad{\'a}mek, H.~Herrlich, J.~Rosick{\'y}, and W.~Tholen.
\newblock On a generalized small-object argument for the injective subcategory
  problem.
\newblock {\em Cah. Topol. G{\'e}om. Diff{\'e}r. Cat{\'e}g.}, 43(2):83--106,
  2002.

\bibitem{zbMATH07461229}
J.~Ad{\'a}mek and J.~Rosick{\'y}.
\newblock Approximate injectivity and smallness in metric-enriched categories.
\newblock {\em J. Pure Appl. Algebra}, 226(6):30, 2022.
\newblock Id/No 106974.

\bibitem{ahs}
Ji{\v{r}}{\'{\i}} Ad{\'a}mek, Horst Herrlich, and George~E. Strecker.
\newblock Abstract and concrete categories: the joy of cats.
\newblock {\em Repr. Theory Appl. Categ.}, 2006(17):1--507, 2006.

\bibitem{ar}
Ji\v{r}\'{\i} Ad\'{a}mek and Ji\v{r}\'{\i} Rosick\'{y}.
\newblock {\em Locally presentable and accessible categories}, volume 189 of
  {\em London Mathematical Society Lecture Note Series}.
\newblock Cambridge University Press, Cambridge, 1994.

\bibitem{zbMATH05265624}
Charalambos~D. Aliprantis and Kim~C. Border.
\newblock {\em Infinite dimensional analysis. {A} hitchhiker's guide.}
\newblock Berlin: Springer, 3rd ed. edition, 2006.

\bibitem{blk}
B.~Blackadar.
\newblock {\em Operator algebras}, volume 122 of {\em Encyclopaedia of
  Mathematical Sciences}.
\newblock Springer-Verlag, Berlin, 2006.
\newblock Theory of $C^*$-algebras and von Neumann algebras, Operator Algebras
  and Non-commutative Geometry, III.

\bibitem{zbMATH03996430}
Bruce Blackadar.
\newblock Shape theory for {{\(C^ *\)}}-algebras.
\newblock {\em Math. Scand.}, 56:249--275, 1985.

\bibitem{zbMATH02131689}
Bruce Blackadar.
\newblock Semiprojectivity in simple {{\(C^*\)}}-algebras.
\newblock In {\em Operator algebras and applications. Proceedings of the
  US-Japan seminar held at Kyushu University, Fukuoka, Japan, June 7--11,
  1999}, pages 1--17. Tokyo: Mathematical Society of Japan, 2004.

\bibitem{zbMATH03416517}
Ola Bratteli.
\newblock Inductive limits of finite dimensional {C}{{\(^*\)}}-algebras.
\newblock {\em Trans. Am. Math. Soc.}, 171:195--234, 1972.

\bibitem{MR2799809}
Nathanial~P. Brown.
\newblock Topological dynamical systems associated to {${\rm II}_1$}-factors.
\newblock {\em Adv. Math.}, 227(4):1665--1699, 2011.
\newblock With an appendix by Narutaka Ozawa.

\bibitem{zbMATH06740635}
Janusz Brzd{\k{e}}k, Dorian Popa, Ioan Ra{\c{s}}a, and Bing Xu.
\newblock {\em Ulam stability of operators}.
\newblock Math. Anal. Appl. Amsterdam: Elsevier/Academic Press, 2018.

\bibitem{bbi}
Dmitri Burago, Yuri Burago, and Sergei Ivanov.
\newblock {\em A course in metric geometry}, volume~33 of {\em Graduate Studies
  in Mathematics}.
\newblock American Mathematical Society, Providence, RI, 2001.

\bibitem{zbMATH06183792}
Valerio Capraro and Tobias Fritz.
\newblock On the axiomatization of convex subsets of {Banach} spaces.
\newblock {\em Proc. Am. Math. Soc.}, 141(6):2127--2135, 2013.

\bibitem{zbMATH05954612}
Jesus M.~F. Castillo.
\newblock The hitchhiker guide to categorical {Banach} space theory. {I}.
\newblock {\em Extr. Math.}, 25(2):103--149, 2010.

\bibitem{2310.15139v1}
Alexandru Chirvasitu.
\newblock (quantum) discreteness, spectrum compactness and uniform continuity,
  2023.
\newblock http://arxiv.org/abs/2310.15139v1.

\bibitem{zbMATH07595194}
Alexandru Chirvasitu and Joanna Ko.
\newblock Monadic forgetful functors and (non-)presentability for
  {{\(C^\ast\)}}- and {{\(W^\ast\)}}-algebras.
\newblock {\em J. Pure Appl. Algebra}, 227(3):24, 2023.
\newblock Id/No 107209.

\bibitem{cp-us}
Alexandru Chirvasitu and Jun Peng.
\newblock Manifolds of lie-group-valued cocycles and discrete cohomology, 2022.
\newblock http://arxiv.org/abs/2211.11429v1.

\bibitem{con_course-2e}
John~B. Conway.
\newblock {\em A course in functional analysis.}, volume~96 of {\em Grad. Texts
  Math.}
\newblock New York etc.: Springer-Verlag, 2nd ed. edition, 1990.

\bibitem{cg-average}
Gustavo Corach and Jos\'{e}~E. Gal\'{e}.
\newblock Averaging with virtual diagonals and geometry of representations.
\newblock In {\em Banach algebras '97 ({B}laubeuren)}, pages 87--100. de
  Gruyter, Berlin, 1998.

\bibitem{dales}
H.~G. Dales.
\newblock {\em Banach algebras and automatic continuity}, volume~24 of {\em
  London Mathematical Society Monographs. New Series}.
\newblock The Clarendon Press, Oxford University Press, New York, 2000.
\newblock Oxford Science Publications.

\bibitem{dael}
H.~Garth Dales, Pietro Aiena, J\"{o}rg Eschmeier, Kjeld Laursen, and George~A.
  Willis.
\newblock {\em Introduction to {B}anach algebras, operators, and harmonic
  analysis}, volume~57 of {\em London Mathematical Society Student Texts}.
\newblock Cambridge University Press, Cambridge, 2003.

\bibitem{zbMATH03577502}
Pierre de~la Harpe and Max Karoubi.
\newblock Repr{\'e}sentations approchees d'un groupe dans une alg{\`e}bre de
  {Banach}.
\newblock {\em Manuscr. Math.}, 22:293--310, 1977.

\bibitem{zbMATH03342104}
Frank DeMeyer and Edward Ingraham.
\newblock {\em Separable algebras over commutative rings}, volume 181 of {\em
  Lect. Notes Math.}
\newblock Berlin-Heidelberg-New York: Springer-Verlag, 1971.

\bibitem{zbMATH07504760}
Ivan Di~Liberti and Ji{\v{r}}{\'{\i}} Rosick{\'y}.
\newblock Enriched locally generated categories.
\newblock {\em Theory Appl. Categ.}, 38:661--683, 2022.

\bibitem{zbMATH03236278}
Adrien Douady.
\newblock Le probl{\`e}me des modules pour les sous-espaces analytiques
  compacts d'un espace analytique donn{\'e}.
\newblock {\em Ann. Inst. Fourier}, 16(1):1--95, 1966.

\bibitem{zbMATH07027314}
Antonio Falc{\'o}, Wolfgang Hackbusch, and Anthony Nouy.
\newblock On the {Dirac}-{Frenkel} variational principle on tensor {Banach}
  spaces.
\newblock {\em Found. Comput. Math.}, 19(1):159--204, 2019.

\bibitem{zbMATH06465637}
Joanna Garbuli{\'n}ska-W{\k{e}}grzyn and Wies{\l}aw Kubi{\'s}.
\newblock A universal operator on the {Gurari{\u{\i}}} space.
\newblock {\em J. Oper. Theory}, 73(1):143--158, 2015.

\bibitem{zbMATH05114904}
Misha Gromov.
\newblock {\em Metric structures for {Riemannian} and non-{Riemannian} spaces.
  {Transl}. from the {French} by {Sean} {Michael} {Bates}. {With} appendices by
  {M}. {Katz}, {P}. {Pansu}, and {S}. {Semmes}. {Edited} by {J}. {LaFontaine}
  and {P}. {Pansu}}.
\newblock Mod. Birkh{\"a}user Class. Basel: Birkh{\"a}user, 3rd printing
  edition, 2007.

\bibitem{zbMATH03101068}
D.~H. Hyers.
\newblock On the stability of the linear functional equation.
\newblock {\em Proc. Natl. Acad. Sci. USA}, 27:222--224, 1941.

\bibitem{zbMATH03046328}
D.~H. Hyers and S.~M. Ulam.
\newblock Approximate isometries of the space of continuous functions.
\newblock {\em Ann. Math. (2)}, 48:285--289, 1947.

\bibitem{zbMATH03074215}
D.~H. Hyers and S.~M. Ulam.
\newblock Approximately convex functions.
\newblock {\em Proc. Am. Math. Soc.}, 3:821--828, 1952.

\bibitem{zbMATH00149467}
Donald~H. Hyers and Themistocles~M. Rassias.
\newblock Approximate homomorphisms.
\newblock {\em Aequationes Math.}, 44(2-3):125--153, 1992.

\bibitem{zbMATH03243151}
J.~R. Isbell.
\newblock Six theorems about injective metric spaces.
\newblock {\em Comment. Math. Helv.}, 39:65--76, 1964.

\bibitem{zbMATH05806299}
Bart Jacobs.
\newblock Convexity, duality and effects.
\newblock In {\em Theoretical computer science. 6th IFIP WG 2.2 international
  conference, TCS 2010, held as a Part of the World Computer Congress (WCC
  2010), Brisbane, Australia, September 20--23, 2010. Proceedings}, pages
  1--19. Berlin: Springer, 2010.

\bibitem{zbMATH06814745}
Bart Jacobs.
\newblock From probability monads to commutative effectuses.
\newblock {\em J. Log. Algebr. Methods Program.}, 94:200--237, 2018.

\bibitem{zbMATH03389730}
B.~E. Johnson.
\newblock Approximate diagonals and cohomology of certain annihilator {Banach}
  algebras.
\newblock {\em Am. J. Math.}, 94:685--698, 1972.

\bibitem{john-coh}
Barry~Edward Johnson.
\newblock {\em Cohomology in {B}anach algebras}.
\newblock Memoirs of the American Mathematical Society, No. 127. American
  Mathematical Society, Providence, R.I., 1972.

\bibitem{zbMATH04063807}
Barry~Edward Johnson.
\newblock Approximately multiplicative maps between {Banach} algebras.
\newblock {\em J. Lond. Math. Soc., II. Ser.}, 37(2):294--316, 1988.

\bibitem{zbMATH05012871}
J{\"u}rgen Jost.
\newblock {\em Postmodern analysis}.
\newblock Universitext. Berlin: Springer, 2005.

\bibitem{zbMATH03289866}
R.~R. Kallman.
\newblock A characterization of uniformly continuous unitary representations of
  connected locally compact groups.
\newblock {\em Mich. Math. J.}, 16:257--263, 1969.

\bibitem{koeth_tvs-1}
Gottfried K\"{o}the.
\newblock {\em Topological vector spaces. {I}}.
\newblock Die Grundlehren der mathematischen Wissenschaften, Band 159.
  Springer-Verlag New York, Inc., New York, 1969.
\newblock Translated from the German by D. J. H. Garling.

\bibitem{lam}
T.~Y. Lam.
\newblock {\em A first course in noncommutative rings}, volume 131 of {\em
  Graduate Texts in Mathematics}.
\newblock Springer-Verlag, New York, second edition, 2001.

\bibitem{lang-fund}
Serge Lang.
\newblock {\em Fundamentals of differential geometry}, volume 191 of {\em
  Graduate Texts in Mathematics}.
\newblock Springer-Verlag, New York, 1999.

\bibitem{zbMATH01216133}
Saunders Mac~Lane.
\newblock {\em Categories for the working mathematician.}, volume~5 of {\em
  Grad. Texts Math.}
\newblock New York, NY: Springer, 2nd ed edition, 1998.

\bibitem{mart-proj}
Mircea Martin.
\newblock Projective representations of compact groups in {$C^*$}-algebras.
\newblock In {\em Linear operators in function spaces ({T}imi\c{s}oara, 1988)},
  volume~43 of {\em Oper. Theory Adv. Appl.}, pages 237--253. Birkh\"{a}user,
  Basel, 1990.

\bibitem{neeb-inf}
Karl-Hermann Neeb.
\newblock Infinite-dimensional groups and their representations.
\newblock In {\em Lie theory}, volume 228 of {\em Progr. Math.}, pages
  213--328. Birkh\"{a}user Boston, Boston, MA, 2004.

\bibitem{ped-aut}
Gert~K. Pedersen.
\newblock {\em {$C^*$}-algebras and their automorphism groups}.
\newblock Pure and Applied Mathematics (Amsterdam). Academic Press, London,
  2018.
\newblock Second edition of [ MR0548006], Edited and with a preface by S\o ren
  Eilers and Dorte Olesen.

\bibitem{zbMATH03425823}
N.~Popescu.
\newblock {\em Abelian categories with applications to rings and modules},
  volume~3 of {\em Lond. Math. Soc. Monogr.}
\newblock Academic Press, London, 1973.

\bibitem{poth}
Kenneth Pothoven.
\newblock Projective and injective objects in the category of {B}anach spaces.
\newblock {\em Proc. Amer. Math. Soc.}, 22:437--438, 1969.

\bibitem{zbMATH03955758}
D.~Pumpl{\"u}n and H.~R{\"o}hrl.
\newblock Separated totally convex spaces.
\newblock {\em Manuscr. Math.}, 50:145--183, 1985.

\bibitem{zbMATH01563366}
D.~Pumpl{\"u}n and H.~R{\"o}hrl.
\newblock Convexity theories. {V}: {Extensions} of absolutely and totally
  convex modules.
\newblock {\em Appl. Categ. Struct.}, 8(3):527--543, 2000.

\bibitem{zbMATH03879553}
Dieter Pumpl{\"u}n and Helmut R{\"o}hrl.
\newblock Banach spaces and totally convex spaces. {I}.
\newblock {\em Commun. Algebra}, 12:953--1019, 1984.

\bibitem{zbMATH04004954}
Dieter Pumpl{\"u}n and Helmut R{\"o}hrl.
\newblock The coproduct of totally convex spaces.
\newblock {\em Beitr. Algebra Geom.}, 24:249--278, 1987.

\bibitem{zbMATH00097447}
Dieter Pumpl{\"u}n and Helmut R{\"o}hrl.
\newblock Totally convex algebras.
\newblock {\em Commentat. Math. Univ. Carol.}, 33(2):205--235, 1992.

\bibitem{zbMATH03618859}
Themistocles~M. Rassias.
\newblock On the stability of the linear mapping in {Banach} spaces.
\newblock {\em Proc. Am. Math. Soc.}, 72:297--300, 1978.

\bibitem{zbMATH03303492}
G.~A. Reid.
\newblock Epimorphisms and surjectivity.
\newblock {\em Invent. Math.}, 9:295--307, 1970.

\bibitem{zbMATH06329568}
Du{\v{s}}an Repov{\v{s}} and Pavel~V. Semenov.
\newblock Continuous selections of multivalued mappings.
\newblock In {\em Recent progress in general topology III. Based on the
  presentations at the Prague symposium, Prague, Czech Republic, 2001}, pages
  711--749. Amsterdam: Atlantis Press, 2014.

\bibitem{riehl_ht}
Emily Riehl.
\newblock {\em Categorical homotopy theory}, volume~24 of {\em New Mathematical
  Monographs}.
\newblock Cambridge University Press, Cambridge, 2014.

\bibitem{rob}
Alain Robert.
\newblock {\em Introduction to the representation theory of compact and locally
  compact groups}, volume~80 of {\em London Mathematical Society Lecture Note
  Series}.
\newblock Cambridge University Press, Cambridge-New York, 1983.

\bibitem{rr_tvs}
A.~P. Robertson and Wendy Robertson.
\newblock {\em Topological vector spaces}, volume~53 of {\em Cambridge Tracts
  in Mathematics}.
\newblock Cambridge University Press, Cambridge-New York, second edition, 1980.

\bibitem{zbMATH07469564}
J.~Rosick{\'y}.
\newblock Are {Banach} spaces monadic?
\newblock {\em Commun. Algebra}, 50(1):268--274, 2022.

\bibitem{zbMATH07760731}
J.~Rosick{\'y}.
\newblock Enriched purity and presentability in {Banach} spaces.
\newblock {\em Commun. Algebra}, 51(12):5242--5262, 2023.

\bibitem{zbMATH06916415}
J.~Rosick{\'y} and W.~Tholen.
\newblock Approximate injectivity.
\newblock {\em Appl. Categ. Struct.}, 26(4):699--716, 2018.

\bibitem{zbMATH01022658}
Walter Rudin.
\newblock {\em Real and complex analysis.}
\newblock New York, NY: McGraw-Hill, 3rd ed. edition, 1987.

\bibitem{tak2}
M.~Takesaki.
\newblock {\em Theory of operator algebras. {II}}, volume 125 of {\em
  Encyclopaedia of Mathematical Sciences}.
\newblock Springer-Verlag, Berlin, 2003.
\newblock Operator Algebras and Non-commutative Geometry, 6.

\bibitem{zbMATH00949298}
A.~C. Thompson.
\newblock {\em Minkowski geometry}, volume~63 of {\em Encycl. Math. Appl.}
\newblock Cambridge: Cambridge University Press, 1996.

\bibitem{MR3462046}
Vlad Timofte and Aida Timofte.
\newblock Generalized {D}ini theorems for nets of functions on arbitrary sets.
\newblock {\em Positivity}, 20(1):171--185, 2016.

\bibitem{zbMATH05306371}
C{\'e}dric Villani.
\newblock {\em Optimal transport. {Old} and new}, volume 338 of {\em
  Grundlehren Math. Wiss.}
\newblock Berlin: Springer, 2009.

\bibitem{wil-top}
Stephen Willard.
\newblock {\em General topology}.
\newblock Dover Publications, Inc., Mineola, NY, 2004.
\newblock Reprint of the 1970 original [Addison-Wesley, Reading, MA;
  MR0264581].

\end{thebibliography}
%\bibliographystyle{plain}

\Addresses

\end{document}